\documentclass[12pt]{article}
\usepackage{amsmath,amssymb,amsthm,mathrsfs,hyperref,theoremref,lipsum,lineno,enumitem,lmodern,color}

\def\re{{\Re e\,}}

\def\p{{\mathfrak{p}}}

\newcommand{\Bigoplus}{\mathop{\bigoplus}\limits}
\theoremstyle{plain}
\newtheorem*{thm*}{Theorem}
\newtheorem{thm}{Theorem}[section]
\newtheorem{conj}{Conjecture}

\newtheorem{pro}{Proposition}[section]

\theoremstyle{definition}

\newtheorem{dfn}{Definition}[section]

\numberwithin{equation}{section}

\title{Chebotarev-Sato-Tate distribution for abelian surfaces potentially of $\rm{GL}_2$-type}

\author{Mohammed Amin Amri} 
 \newcommand{\Addresses}{{
  \bigskip

  Mohammed Amin.~Amri, \textsc{Higher School of Education and Training, Ibn Tofail University,Kenitra, Morocco}\par\nopagebreak
  \textit{E-mail address}, Mohammed Amin~Amri: \texttt{amri.amine.mohammed@gmail.com}

}}   
\begin{document}

\maketitle
\begin{abstract}
In this paper, we state a hybrid Chebotarev-Sato-Tate conjecture for abelian varieties and we prove it in several particular cases using current potential automorphy theorems.
\end{abstract}
\section{Introduction}

The Sato-Tate conjecture has a long history which goes back to the second half of the twentieth century. It states that the Frobenius angels $\theta_p$ of a non-CM elliptic curve $E$ defined over $\mathbb{Q}$, are equidistributed in the interval $[0,\pi]$ with respect to the measure $\mu_{ST}:=\frac{2}{\pi}\sin^{2}\theta d\theta$ when $p$ tends to infinity.

The first reduction of the conjecture was initiated by Tate himself in \cite[\S 4]{Ta65} about 1965, wherein he outlined a heuristic that reduces essentially the conjecture to the Hadamard-de la Vall\`ee Poussin's argument for the symmetric power $L$-functions of the Galois representation attached to the elliptic curve. Soon afterwards, Serre in \cite{Serre2} deduced this observation as a consequence of his characterization of the problem of equidistribution of a sequence on the space of conjugacy classes of a compact group with respect to its Haar measure, in terms of analytic continuation and nonvanishing of some $L$-functions associated to the group in question, using the standard Tauberian theory.

In retrospect, motivated by the Shimura-Taniyama conjecture proved by Wiles and his successors, Langlands \cite{Lang70} proposed a conjectural approach to attack the problem of analytic continuation which have later inspired so much generalization of the conjecture (see for instance \cite[\S 13]{Serre3}). Roughly speaking, he predicts that the Galois representation attached to the elliptic curve can be associated to a certain cuspidal automorphic representation. Hence, the required analytic properties on its $m$-th symmetric power $L$-functions could be deduced from standard analytic proprieties of the $L$-functions associated to some cuspidal automorphic representation of ${\rm GL}_{m+1}(\mathbb A_{\mathbb Q})$. For more historical details the reader could consult \cite{Cl}.

The development of the theory of potential automorphy which was initiated by R.Taylor in his seminal papers \cite{Tay02,Tay03}, culminated Clozel, Harris, Shepherd-Barron, Taylor in \cite{CHT,HSBT10} proving the Sato-Tate conjecture. On the basis of these insights, M.R. Murty and V.K. Murty \cite{Murty} derived a hybrid Chebotarev-Sato-Tate theorem that shows that the Artin symbols of any given solvable Galois extension and the Frobenius angles of a non-CM elliptic curve over a totally real number field are equidistributed. In light of the recent progress in potential automorphy theory, especially the work of Barnet-Lamb et al. \cite{BGG11,BGGT14}, Wong \cite{Wong} extended Murty-Murty's work to a pair of non-CM Hilbert modular forms that are not twist-equivalent. Moreover, he relaxed the severe restriction of solvability. 

In the present paper, we aim to tackle a similar paradigm to that in \cite{Murty,Wong} in the context of abelian surfaces potentially of $\rm{GL}_2$-type, which we believe would yield deep arithmetical consequences. Roughly, we shall prove the following theorem in this direction (see \thref{thm1,thm2,thm3,thm4}).
\begin{thm}\thlabel{mainTh}
Let $A$ be an abelian surface over a totally real number field $F$ potentially of $\rm{GL}_2$-type, and let $K/F$ be a Galois extension of finite degree with Galois group $G\cong G_1\times G_2$, such that $G_1$ is abelian and $K^{G_1}/F$ is totally real, assume further that $K$ is linearly disjoint over $F$ from some at most quadratic extension of $F$, then \thref{CST} holds for $A$.
\end{thm}
The paper is organized in the following manner. Section~\ref{Sec1} is subdivided into three subsections in which we present some preliminaries. We begin in Subsection~\ref{Sec1.1} by introducing the relevant material that will be needed to state the conjecture as well as some notation and terminology to be used throughout the paper. In Subsection~\ref{Sec1.2} we outline the strategy adopted to prove \thref{mainTh}, we then state and prove a minor variant of the current potential automorphy results adapted to our purpose. Lastly, in Subsection~\ref{Sec1.3} we give a brief exposition from \cite{FKRS} of the classification of the Sato-Tate group of abelian surfaces. In Section~\ref{Sec2} our main results are stated and proved.
\paragraph*{Notations.}  Let $A$ be an abelian variety over a number field $F$ we denote by  $\mathrm{End}_F(A)$ (resp. $\mathrm{End}^0_F(A)$) the ring of  endomorphisms  of $A$  defined over $F$ (resp. the  endomorphism  algebra $\mathrm{End}_F(A)\otimes_{\mathbb{Z}} \mathbb{Q}$). If $K/F$ is an extension, we denote by $A_K$ the base change of $A$ to $K$. Throughout, the paper we let $L$ stand for the minimal extension of $F$ such that $\mathrm{End}^0_{L}(A)=\mathrm{End}^0_{\overline{\mathbb Q}}(A)$. We record the following notations that will be used in the text.
\begin{itemize}
\item ${\rm Irr}(G)$ the set of irreducible characters of a group $G$.
\item ${\rm U}(n)=\lbrace M\in {\rm GL}_n(\mathbb{C})\;|\; ^{t}\overline{M}M=I_n\rbrace$ the unitary group.
\item ${\rm SU}(n)=\lbrace M\in {\rm U}(n)\;|\; \det(M)=1\rbrace$ the special unitary group.
\item ${\rm GSp}_{2n}(F)=\lbrace M\in {\rm GL}_{2n}(F)\;|\; ^{t}M\Omega M=\nu(M)\Omega\;\text{for some}\; \nu(M)\in F^{\times}\rbrace$ the symplectic similitude group, where $F=\mathbb{C}$ or $\mathbb{Q}_{\ell}$ and $\Omega=\begin{pmatrix}
0 & I_n \\
-I_n & 0
\end{pmatrix}$.
\item ${\rm USp}(2n)=\lbrace M\in {\rm GSp}_{2n}(\mathbb{C})\;|\; \nu(M)=1\rbrace$ the symplectic group.
\item $C_n$, $D_{2n}$ and $\mu_{2n}$ are respectively the cyclic group, the Dihedral group and the group of $2n$-th roots of unity.
\end{itemize}

\section{Preliminaries}\label{Sec1}
\subsection{Statement of the conjecture}\label{Sec1.1}
We begin by recalling the definition of the Sato-Tate group and introduce a definition of the ``Chebotarev-Sato-Tate group" following closely the exposition of \cite[Section 3.2]{Sut} see also \cite[Chapter 8]{Serre}.\\
Let $A$ be a polarized abelian variety of dimension $g$ over a number field $F$ and  let $K/F$ be a finite Galois extension. For a fixed rational prime $\ell$,  we define the $\ell$-adic Tate module $T_\ell(A):= \varprojlim_{n} A[\ell^n]$ to be a free $\mathbb Z_\ell$-module of rank $2g$, and the $\ell$-adic rational Tate module $V_\ell(A):=T_\ell(A)\otimes_{\mathbb{Z}}\mathbb{Q}$ to be a $\mathbb{Q}_\ell$-vector space of dimension $2g$. Fix a symplectic basis of the first singular homology group $H_1(A(\mathbb{C}),\mathbb{Q})$ for a fixed embedding $F\hookrightarrow\mathbb{C}$. We identify the Weil pairings on the rational Tate module with the cup product pairing in \'etal homology and singular homology by means of the following identifications (the first two ones follow from \cite[Theorem 15.1]{Milne} while the second ones follow by exploiting the interpretation of $A(\mathbb{C})$ as a complex torus $\mathbb{C}^{g}/\Lambda$ where $\Lambda$ is the lattice $H_1(A(\mathbb{C}),\mathbb{Z})$)
$$
V_\ell(A)\cong H_{1,\text{\'et}}(A_{\overline{\mathbb{Q}}},\mathbb{Q}_\ell)\cong H_{1,\text{\'et}}(A_{\mathbb{C}},\mathbb{Q}_\ell)\cong H_1(A(\mathbb{C}),\mathbb{Q}_\ell)\cong H_1(A(\mathbb{C}),\mathbb{Q})\otimes \mathbb{Q}_\ell.
$$
Accordingly, the symplectic basis for $H_1(A(\mathbb{C}),\mathbb{Q})$ gives a symplectic basis for $V_\ell(A)$ and the action of the Galois group $G_F: = \mathrm{Gal}(\overline{F}/F)$  give rises to a $\ell$-adic Galois representation
$$
\rho_{\ell,A} : G_F\longrightarrow \mathrm{GSp}_{2g}(\mathbb{Q}_\ell).
$$
Let $ G_\ell$ be the image of $\rho_{\ell,A}$ and $G^{\mathrm{Zar}}_\ell$ is its Zariski closure inside $\mathrm{GSp}_{2g}(\mathbb{Q}_\ell)$  (the $\ell$-adic monodromy group of $A$). Let $\nu : G_{\ell}^{\mathrm{Zar}}\longrightarrow \mathbb{G}_m$ be the similitude character on $G_{\ell}^{\mathrm{Zar}}$ and we let $G_{\ell}^{1,\mathrm{Zar}}$ denote its kernel.

For convenience we will fix once and for all an isomorphism $\imath : \overline{\mathbb{Q}}_\ell\xrightarrow{\sim}{} \mathbb{C}$ for all $\ell$ (to not overload the notation we will omit $\ell$ from the notation). This enable us to embed $\mathrm{GSp}_{2g}(\mathbb{Q}_\ell)$ into $\mathrm{GSp}_{2g}(\mathbb{C})$. Setting $G^{\mathrm{Zar}}_{\ell,\imath}(\mathbb{C}):=G^{\mathrm{Zar}}_{\ell}(\mathbb{Q}_\ell)\otimes_{\mathbb{Q}_\ell,\imath}\mathbb{C}$ and $G^{1,\mathrm{Zar}}_{\ell,\imath}(\mathbb{C}):=G^{1,\mathrm{Zar}}_{\ell}(\mathbb{Q}_\ell)\otimes_{\mathbb{Q}_\ell,\imath}\mathbb{C}$. 
\begin{dfn}
We define the \textit{Sato-Tate group} $ST_A\subset \mathrm{USp}(2g)$ of $A$ as the maximal compact subgroup of the complex Lie group $G^{1,\mathrm{Zar}}_{\ell,\imath}(\mathbb{C})$. We let $ST$ denote $ST_A\times \mathrm{Gal}(K/F)$, what we shall referring to the  \textit{Chebotarev-Sato-Tate group}.
\end{dfn}
Let $S$ be the set of primes which lie over $\ell$ and those primes at which $A$ has bad reduction or at which $K/F$ ramifies. For $\p\notin S$, we let $\mathrm{Frob}_\p$ denote an arithmetic Frobenius at $\p$ in $G_F$ (here and subsequently throughout the paper), let $g_\p$ be its image under the map
$$G_F \relbar\joinrel\xrightarrow{\rho_{\ell,A}}{} G_\ell \lhook\joinrel\relbar\joinrel\relbar\joinrel\rightarrow G^{\mathrm{Zar}}_\ell(\mathbb{Q}_\ell)\lhook\joinrel\relbar\joinrel\relbar\joinrel\rightarrow G^{\mathrm{Zar}}_{\ell,\imath}(\mathbb{C}).$$
Let $q_{\p}$ be the order of the residue field $\mathbb{F}_{\p}$ at $\p$. From the main theorem of \cite{Ta66} together with part (a) of Theorem 2 in {\it ibid} we see that the matrix  $g_\p q_{\p}^{-1/2}$ is diagonalizable and has eigenvalues of absolute value $1$, it therefore lies in a compact subgroup of $G^{1,{\rm Zar}}_{\ell,\iota}(\mathbb{C})$. This subgroup must be conjugate to a subgroup of the maximal compact subgroup $ST_A$. It then follows that $g_\p q_{\p}^{-1/2}$ should be conjugate to an element of $ST_A$. Define $x_{\p}$ to be the conjugacy class of $g_\p q_{\p}^{-1/2}  \times \sigma_\p$ in the Chebotarev-Sato-Tate groupe $ST$, where $\sigma_\p:= \left(\frac{K/F}{\p}\right)$ denote the Artin symbol relative to the extension $K/F$. The Haar measure $\mu$ of $ST$ is a product of the Haar measure $\mu_{ST}$ of $ST_A$ and the discrete measure on ${\rm Gal}(K/F)$. We let $X$ stands for the set of conjugacy classes of $ST$. One may conjecture.

\begin{conj}\thlabel{CST}
The sequence $\lbrace x_\p\rbrace_{\p\not\in S}$, where the primes $\p$ are ordered by norm, is equidistributed on $X$ with respect to the pushforward of the Haar measure $\mu$ on $ST$ to $X$.
\end{conj}

\subsection{Strategy of the proof}\label{Sec1.2}
In this subsection, we record the relevant facts that will be needed in various stages of the approach that we shall follow to prove \thref{CST}. We start by quoting the following result of Serre \cite[Appendix to Chap. I]{Serre2} which will play a central role in the sequel.
\begin{thm}[Serre 1968]
Let $G$ be a compact group, and ${\rm Conj}(G)$ its space of conjugacy classes.  Suppose that for each prime $p$, we associate a conjugacy class $x_p$ in ${\rm Conj}(G)$. Then the sequence $\{x_p\}_{p}$ is equidistributed with respect to the image of the Haar measure of $G$ in ${\rm Conj}(G)$ if and only if, for every irreducible nontrivial representation $\rho$ of $G$, the $L$-function
 $$
 L(s,\rho):=\prod_{p}\det(1-\rho(x_p)p^{-s})^{-1}
 $$
extends to an analytic function for $\re(s)\ge 1$, and does not
vanish there.
\end{thm}
Keeping the notations of the preceding subsection. In view of Serre's theorem,  proving the equidistribution of the sequence $\{x_\p\}_{\p}$ over $X$ with respect to the image of the Haar measure of $ST$, turns out to prove that for every irreducible nontrivial representation $\rho$ of $ST$, the $L$-function $$L(s,\rho)=\prod_{\mathfrak{p}}\det(1-\rho(x_\mathfrak{p})q_{\mathfrak{p}}^{-s})^{-1}$$
can be continued analytically and dose not vanish on the region $\re(s)\ge 1$ . For brevity, the property of analytic continuation and non-vanishing on the region $\re(s)\ge 1$ of a $L$-function, will be called ``invertibility''.\\
 Notice that, since $ST_{A}$ is a maximal compact subgroup of the complex Lie group $G^{1,Zar}_{\ell,\imath}(\mathbb{C})$. By \cite[Corollary 2, pp.77]{Serre1} any irreducible representation $\rho$ of $ST_{A}$  may extends uniquely to an irreducible representation $\tilde{\rho}$ of $G^{1,Zar}_{\ell,\imath}(\mathbb{C})$. In view of the exact sequence
$$1 \relbar\joinrel\xrightarrow{}{} G^{1,\mathrm{Zar}}_{\ell} \lhook\joinrel\relbar\joinrel\rightarrow G^{\mathrm{Zar}}_{\ell}\relbar\joinrel\xrightarrow{\nu}{}\mathbb{G}_m \joinrel\relbar\joinrel\rightarrow 1.$$
The group $G^{\mathrm{Zar}}_{\ell}$ is the quotient of $G^{1,\mathrm{Zar}}_{\ell} \times \mathbb{C}^\times$ with a kernel of order two, generated by the element $(-1,-1)$. We then twist $\tilde{\rho}$  by a character $\chi_\omega : z\mapsto z^\omega$ of  $\mathbb{C}^{\times}$ for some $\omega\in\mathbb{Z}$ such that $\tilde{\rho}\otimes\chi_\omega$ is trivial on $\left<(-1,-1)\right>$. By composing $\tilde{\rho}\otimes\chi_\omega$ with the quotient map we obtain an algebraic representation $R(\omega)$ of $G^{\mathrm{Zar}}_{\ell}$. Such as we get a weakly compatible system of Galois representations $R(\omega)\circ\rho_{A,\ell}$ which satisfies 
$$L(s,\rho\otimes\xi)=L(s+\omega/2,R(\omega)\circ\rho_{A,\ell}\otimes\xi)\quad\text{for any} \quad\xi\in\mathrm{Irr}(\mathrm{Gal}(K/F)).$$
The best strategy currently known for checking Serre's criterion for such $L$-functions is to use potential automorphy results \cite[Theorem 7.1.10]{ACC} \cite[Theorem 4.5.1]{BGGT14}  combined with Brauer's theorem \cite{Brauer} and Arthur--Clozel's \cite{AC90} solvable base change, together with known functoriality of $\mathrm{GL}(n)\times\mathrm{GL}(1)$ \`a la \cite{Tay02,HSBT10}. For our purpose we shall need the following simple variant of \cite[Theorem 4.5.1]{BGGT14} (in the totally real case) and \cite[Theorem 7.1.10]{ACC} (in the CM case). 
\begin{thm}\thlabel{Potaut}
\sloppy Suppose that $F$ is a CM (resp. totally real) field, $\mathcal{R}_1=(M_1,S_1,\{Q_{1,v}(X)\},\{r_{1,\lambda}\}, \{\{0,1\}\})$ and $\mathcal{R}_2=(M_2,S_2,\{Q_{2,v}(X)\},\{r_{2,\lambda}\}, \{\{0,1\}\})$ are a pair of strongly irreducible (resp. strongly irreducible, totally odd and polarized) rank $2$ weakly compatible system of $\ell$-adic representations of $G_F$, and $K/F$ is a finite extension of CM (resp. totally real) fields. Then, if $n_1,n_2$ are positive integers, there exists a finite CM (resp. totally real) extension $L'/F$  containing $K$ such that $L'/\mathbb{Q}$ is Galois and the weakly compatible systems $\mathrm{Sym}^{n_1}\mathcal{R}_{1|L'}$ and $\mathrm{Sym}^{n_2}\mathcal{R}_{2|L'}$ are automorphic. 
\end{thm}

\begin{proof}
Observe that since the Galois closure of a CM (resp. totally real) field over $\mathbb{Q}$ is CM (resp. totally real) one may assume, by replacing $K$ with its Galois closure over $\mathbb{Q}$, that $K/\mathbb{Q}$ is Galois.\\
Let us first tackle the CM case. Consider the restrictions $\mathcal{R}_{1|K}$ and $\mathcal{R}_{2|K}$ of the two weakly compatible systems $\mathcal{R}$ and $\mathcal{R}'$ from $G_F$ to $G_K$ respectively. Notice that, $\mathcal{R}_{|K}$ and $\mathcal{R}'_{|K}$ still strongly irreducible with Hodge-Tate multiset equals to  $\lbrace0,1\rbrace$. By \cite[Theorem 7.1.10]{ACC}, there exists a finite CM extension $L'/K$ with $L'/\mathbb{Q}$ Galois, such that the weakly compatible systems  $\mathrm{Sym}^n_{1}\mathcal{R}_{1|L'}$ and $\mathrm{Sym}^n_{2}\mathcal{R}_{2|L'}$ are automorphic.\\
Next, one turns to the totally real case. Replacing each of the fields $M_i$ associated to $\mathcal{R}_i$ by their compositum, we may assume that $M=M_i$ for $i=1,2$. Let $\mathcal{L}$ be the intersection of the two sets of rationals primes of density one provided by applying \cite[Lemma 7.1.3]{ACC} to the two weakly compatible systems $\mathcal{R}_{1|K}$ and $\mathcal{R}_{2|K}$. It follows from \cite[Lemma A.1.7]{BGGT14} that $\mathcal{L}$  also has density one. Then, for all $\lambda|\ell\in\mathcal{L}$ place of $M$, one has ${\rm SL}_{2}(\mathbb{F}_{\ell})\subset \overline{r}_{i,\lambda}(G_{K})$ where $\overline{r}_{i,\lambda}$ stands for the semi-simplified reduction of $r_{i,\lambda}$ for $i=1,2$. Now we choose a totally imaginary quadratic CM extension $F'$ of $K$ which is unramified above $\mathcal{L}$, Galois over $\mathbb{Q}$ and linearly disjoint from $\overline{K}^{{\rm ker}\,\overline{r}_{i,\lambda|G_K}}$ over $K$ for $i=1,2$. Note that, for all $\lambda|\ell\in\mathcal{L}$ place of $M$, we have ${\rm SL}_{2}(\mathbb{F}_{\ell})\subset \overline{r}_{i,\lambda}(G_{K})=\overline{r}_{i,\lambda}(G_{F'})$, since ${\rm SL}_{2}(\mathbb{F}_{\ell})$ is perfect, it then follows that ${\rm SL}_{2}(\mathbb{F}_{\ell})\subset\overline{r}_{i,\lambda}(G_{F'(\zeta_{\ell})})$  where $\zeta_{\ell}$ is a primitive $\ell$-th root of $1$. Hence, for $\lambda|\ell\in \mathcal{L}$ the representations $\mathrm{Sym}^{n}\overline{r}_{i,\lambda|G_{F'(\zeta_{\ell})}}$ are irreducible for all $n\le \ell-1$. After removing finitely many primes from $\mathcal{L}$ one may then suppose that
\begin{itemize}
\item $\ell\in \mathcal{L}$ implies $\ell$ is unramified in $K$ and $\ell$ lies below no element of any set $S_i'$ of bad primes for $\mathcal{R}_{i|K}$.
\item If $v|\ell$ is a prime of $K$ and $\lambda|\ell$, each representation $\mathrm{Sym}^{n_i}r_{i,\lambda|G_K}$ is crystalline at $v$.
\item If $\ell\in \mathcal{L}$ then the Hodge-Tate multiset of each $\mathrm{Sym}^{n_i}r_{i,\lambda|G_K}$ equals to $\lbrace 0,\dots,n_i\rbrace$.
\item  For all $\lambda|\ell\in\mathcal{L}$, each $\mathrm{Sym}^{n_i}\overline{r}_{i,\lambda|G_{F'(\zeta_{\ell})}}$ is irreducible.
\end{itemize}
From \cite[Lemma 1.4.3]{BGGT14}, the former three points imply that $\mathrm{Sym}^{n_i}r_{i,\lambda|G_K}$ is potentially diagonalizable. Thus, the representations $\mathrm{Sym}^{n_i}r_{i,\lambda|G_K}$ satisfy the hypotheses of \cite[Theorem 4.5.1]{BGGT14} for any $\lambda|\ell\in\mathcal{L}$. So, there  exists a finite CM Galois extension $F''/F'$ over which the representations $\mathrm{Sym}^{n_i}r_{i,\lambda}$ become cuspidal automorphic. Let $L'$ be the maximal totally real subfield of $F''$. It may be deduced from \cite[Lemma 2.2.2]{BGGT14} that the representations $\mathrm{Sym}^{n_i}r_{i,\lambda}$ are cuspidal automorphic over $L'$, which achieve the proof of the theorem.
 
\end{proof} 
\subsection{Classification of the Sato-Tate group}\label{Sec1.3}
For the convenience of the reader we shall discuss briefly the classification of the Sato-Tate group in terms of the \emph{Galois type} of the endomorphism group after Fit\'e et al.  in \cite{FKRS} thus making our exposition self-contained. Let us first recall the definition of the Galois type of the endomorphism group.

Given a polarized abelian surface $A$ over a number field $F$. The Galois type associated to $A$ is the isomorphism class of the pair $[\mathrm{Gal}(L/F),\mathrm{End}(A_L)_{\mathbb{R}}]$ where $\mathrm{End}(A_L)_{\mathbb{R}}$ is the real endomorphism algebra $\mathrm{End}(A_L)\otimes\mathbb{R}$ and $L/F$ is the smallest field extension for which $\mathrm{End}(A_L)=\mathrm{End}(A_{\overline{\mathbb Q}})$, with $\mathrm{Gal}(L/F)$ acts on $\mathrm{End}(A_L)_{\mathbb{R}}$ via its action on the coefficients of the rational maps defining each element of $\mathrm{End}(A_L)$. 

The Albert's classification of division algebra with a positive involution together with Shimura's work \cite{Sh}, yield that the real endomorphism algebra of $A$ is one of the six possibilities displayed bellow. In \cite{FKRS} they showed that these six possibilities are in one-to-one correspondence with the six connected Lie subgroups of $\mathrm{USp}(4)$. Further, the Sato-Tate group of $A$ is uniquely determined by the Galois type and vice versa.
\begin{itemize}
\item[\textbf{(A)}]   $\mathrm{End}(A_L)_{\mathbb{R}}=\mathbb{R}$ : which is the generic case, corresponding to $ST_{A}^{0}=\mathrm{USp}(4)$.

\item[\textbf{(B)}]  $\mathrm{End}(A_L)_{\mathbb{R}}=\mathbb{R}\times\mathbb{R}$ : which occurs when $A_L$ is either isogenous to a product of non-isogenous elliptic curves without CM or simple with real multiplication, corresponding to $ST_{A}^{0}=\mathrm{SU}(2)\times\mathrm{SU}(2)$.

\item[\textbf{(C)}]  $\mathrm{End}(A_L)_{\mathbb{R}}=\mathbb{R}\times\mathbb{C}$ : which occurs when $A_L$ is isogenous to a product of two elliptic curves, one with CM and the other one without CM, corresponding to $ST_{A}^{0}=\mathrm{SU}(2)\times\mathrm{U}(1)$.
\item[\textbf{(D)}]  $\mathrm{End}(A_L)_{\mathbb{R}}=\mathbb{C}\times\mathbb{C}$ : which occurs when $A_L$ is either isogenous to a product of non-isogenous elliptic curves with CM or simple with CM by a quadratic field, corresponding to $ST_{A}^{0}=\mathrm{U}(1)\times\mathrm{U}(1)$.
\item[\textbf{(E)}]   $\mathrm{End}(A_L)_{\mathbb{R}}=M_2(\mathbb{R})$ : which occurs when $A_L$ is either isogenous to the square of an elliptic curve without CM or simple with QM (quaternionic multiplication), corresponding to $ST_{A}^{0}=\mathrm{SU}(2)$.
\item[\textbf{(F)}]   $\mathrm{End}(A_L)_{\mathbb{R}}=M_2(\mathbb{C})$ : which occurs when $A_L$ is isogenous to the square of an elliptic curve with CM, corresponding to $ST_{A}^{0}=\mathrm{U}(1)$.
\end{itemize}

On the other hand, they showed that there are a total of 52 distinct Galois types of an abelian surface over a number field, corresponding to 52 distinct Sato-Tate groups. Of these, 35 arise for abelian surfaces defined over a totally real number field and 13 occur only among abelian surfaces which are potentially of $\rm{GL}_2$-type in the sens of \thref{defGL} below, namely
\begin{itemize}
\item[$\bullet$] $\textbf{B}[C_1]$ and $\textbf{B}[C_2]$ corresponding to $ST_A=\rm{SU}(2)\times\rm{SU}(2)$ and $N(\rm{SU}(2)\times\rm{SU}(2))$,
\item[$\bullet$]$\textbf{C}[C_2]$ corresponding to $ST_A=N(\rm{SU}(2)\times\rm{U}(1))$,
\item[$\bullet$] $\textbf{E}[C_n]$ with $n=1,3,4,6$, and $\textbf{E}[C_2,\mathbb{R}\times\mathbb{R}]$ corresponding to $ST_A=E_1, E_3, E_4,E_6$ and $E_2$,
\item[$\bullet$] $\textbf{E}[C_2,\mathbb{C}]$ and $\textbf{E}[D_n]$ with $n=2,3,4,6$ corresponding to $ST_A=J(E_1), J(E_2), J(E_3), J(E_4)$ and $J(E_6)$. 
\end{itemize}
These are the only cases we will consider.
\section{Equidistribution results for abelian surfaces potentially of $\mathrm{GL}_2$-type}\label{Sec2}
In this section, we shall address \thref{CST} in the above displayed cases of abelian surfaces over a totally real number field. The approach adopted is essentially an adaptation of the framework from \cite{Johansson,Noah} and \cite{Wong} which is in turn inspired from \cite{Tay02,HSBT10,Harris} and \cite{Murty} respectively. We begin by recalling a non-standard definition \cite[Definition 18]{Johansson} of abelian surfaces of $\mathrm{GL}_2$-type and reviewing some standard facts on their rational Tate module.
\begin{dfn}\thlabel{defGL}
An abelian variety $A$ over a number field $F$ is termed of $\mathrm{GL}_2$-type if it is isogenous over $F$ to the product of simple abeilan varieties $A_1\times A_2\times\cdots\times A_r$ defined over $F$ each of them whose endomorphism algebras $\mathrm{End}^0_F(A_i)$ contain a number field $P_i$ of maximal dimension $\dim A_i$.
\end{dfn}
Let $A$ be an abelian surface over a number field $F$ potentially of $\mathrm{GL}_2$-type, that is, there exists a number field $P$ of degree $2$, an injection $P\hookrightarrow \mathrm{End}^0_{\overline{\mathbb{Q}}}(A)$, and we have two cases corresponding to whether $A$ is simple and $\mathrm{End}^0_{\overline{\mathbb{Q}}}(A)$ is a  real quadratic field, or $A$ is isogenous to a product of nonisogenous elliptic curves $E_1\times E_2$ and $\mathrm{End}^0_{\overline{\mathbb{Q}}}(A)=\mathbb{Q}\times\mathbb{Q}$.\\ 
When $A$ is simple, for any prime $\ell$, the Tate module $V_\ell(A)$ is free of rank two over $P_\ell=P\otimes_{\mathbb{Q}}\mathbb{Q}_{\ell}$. 
Thus the action of the absolute Galois group $G_{F}$ defines a two-dimensional Galois representation $\rho_{A,\lambda} : G_{F}\longrightarrow \mathrm{GL}(V_\lambda(A))$ where $V_\lambda(A)=V_\ell(A)\otimes_{P_\ell}  P_{\lambda}$ 
and $\lambda$ runs over all primes of $P$ dividing $\ell$.\\
We then have $\mathcal{R}=(V_\lambda(A))_\lambda$ a rank $2$ weakly compatible system of $\lambda$-adic representations of $G_{F}$ defined over $P$ whose exceptional set is the set of prime numbers at which $A$ has bad reduction, which is strongly irreducible and regular \cite[lem. 3.1-3]{Ribet}. 

\subsection{Abelian surfaces of type $\mathbf{B}$}
In this section we envisage the case of abelian surfaces of absolute type ${\bf B}$ and we prove the following equidistribution result.
\begin{pro}\thlabel{thm1}
Let $A$ be an abelian surface over a totally real number field $F$ of type $\mathbf{B}$ and let $K/F$ be a Galois extension of finite degree with Galois group $G\cong G_1\times G_2$, such that $G_1$ is abelian and $K^{G_1}/F$ is totally real, moreover, we assume that $L$ and $K$ are linearly disjoint over $F$, then \thref{CST} holds for $A$.
\end{pro}
\begin{proof}
According to \S 2.3 there are only  two cases to consider.\\
\textbf{Case 1. $A$ of Galois type $\mathbf{B}[C_1]$.}\\ 
As mentioned in \S 2.3 the corresponding  Chebotarev-Sato-Tate group of $A$ is $ST=ST_A\times G$ with $ST_{A}=\mathrm{SU}(2)\times\mathrm{SU}(2)$. Let $\rho_{m,n}$ be a nontrivial irreducible representation of $ST_{A}$ then it is of the form $\rho_{m,n}:=\mathrm{Sym}^m(\mathrm{St})\otimes\mathrm{Sym}^n(\mathrm{St})$ where $\mathrm{St}$ denotes the standard two-dimensional representation of $\mathrm{GL}(2)$, thus the nontrivial irreducible representations of the Chebotarev-Sato-Tate group $ST$ are of the form $\rho_{m,n}\otimes\xi$ with $\xi\in\mathrm{Irr}(G)$. One then extend $\rho_{m,n}$ to a representation $\tilde{\rho}_{m,n}$ of $G^{\mathrm{Zar}}_\ell$. To be specific, since $G^{\mathrm{Zar}}_\ell\subset \mathrm{GL}(2)\times\mathrm{GL}(2)$ we obtain $\tilde{\rho}_{m,n}$ by restricting $\mathrm{Sym}^m(\mathrm{St})\otimes\mathrm{Sym}^n(\mathrm{St})$ to $G^{\mathrm{Zar}}_\ell$. We distinguish the following two sub-cases.\\
\textbf{Sub-case 1: $P=\mathrm{End}^0_{\overline{\mathbb{Q}}}(A)=\mathrm{End}^0_{F}(A)=\mathbb{Q}\times \mathbb{Q}$}\\
In this case $A$ must be isogenous to the product $E_1\times E_2$ of two elliptic curves without CM that do not become isogenous over any finite extension of $F$. Furthermore,  the $\ell$-adic representation $\rho_{A,\ell}$ of $G_F$ attached to $A$ breaks up into a direct sum $\rho_{A,\ell}=\rho_{E_1,\ell}\oplus\rho_{E_2,\ell}$ of the two $\ell$-adic representations $\rho_{E_1,\ell}$ and $\rho_{E_2,\ell}$ of $G_F$ attached to the two elliptic curves $E_1$ and $E_2$ respectively. According to \S 2.2 we need to check that for every $\xi\in\mathrm{Irr}(G)$ the $L$-function
\begin{eqnarray*}
L(s,\rho_{m,n}\otimes\xi)
&=& L\left(s+\frac{m+n}{2},\tilde{\rho}_{m,n}\circ\rho_{A,\ell}\otimes\xi\right)\\
&=& L\left(s+\frac{m+n}{2},\mathrm{Sym}^m\rho_{E_1,\ell}\otimes\mathrm{Sym}^n\rho_{E_2,\ell}\otimes\xi\right)
\end{eqnarray*}
is invertible. To this end we shall follow closely the proof of \cite[Theorem 5.4]{Harris}. We look at the two weakly compatible systems of $\ell$-adic representations $(\rho_{E_1,\ell})_{\ell}$ and $(\rho_{E_2,\ell})_{\ell}$  of $G_F$ which are strongly irreducible (by virtue of Serre's open image theorem) , totally odd, essentially self-dual and with Hodge-Tate multiset $\lbrace 0,1\rbrace$. We may apply \thref{Potaut} to infer that there exists a totally real Galois extension $L'/F$ containing $K^{G_1}$ such that the two compatible systems $(\mathrm{Sym}^m\rho_{E_1,\ell})_{\ell}$ and $(\mathrm{Sym}^n\rho_{E_2,\ell})_{\ell}$ become cuspidal automorphic after restriction to $L'$. Let $\tilde{G}$ denote temporarily the Galois group of $L'/F$. Write $\xi=\chi\otimes\psi$ for $\chi\in\mathrm{Irr}(G_2)$, $\psi\in \mathrm{Irr}(G_1)$ and inflate $\chi$ to a character of $\tilde{G}$. By Brauer's induction theorem we have 
$$
\chi=\Bigoplus_i a_i\mathrm{Ind}_{H_i}^{\tilde{G}}\chi_i
$$
where the $a_i$ are integers and $\chi_i$ are one-dimensional characters of nilpotent subgroups $H_i$ of $\tilde{G}$. It follows that
\begin{multline*}
L\left(s+\frac{m+n}{2},\mathrm{Sym}^m\rho_{E_1,\ell}\otimes\mathrm{Sym}^n\rho_{E_2,\ell}\otimes\chi\otimes\psi\right)=\\ \prod_{i}L\left(s+\frac{m+n}{2},\mathrm{Sym}^m\rho_{E_1,\ell}\otimes\mathrm{Sym}^n\rho_{E_2,\ell}\otimes\psi\otimes\mathrm{Ind}_{H_{i}}^{\tilde{G}}\chi_i\right)^{a_i}.
\end{multline*}
On the other hand, Frobenius reciprocity gives
\begin{eqnarray*}
\mathrm{Sym}^m\rho_{E_1,\ell}\otimes\mathrm{Sym}^n\rho_{E_2,\ell}\otimes\psi\otimes\mathrm{Ind}_{H_i}^{\tilde{G}}\chi_i 
&\cong&\mathrm{Ind}_{H_i}^{\tilde{G}}((\mathrm{Sym}^m\rho_{E_1,\ell}\otimes\mathrm{Sym}^n\rho_{E_2,\ell}\otimes\psi)_{|L'^{H_{i}}}\otimes\chi_i)\\
&\cong&\mathrm{Ind}_{H_i}^{\tilde{G}}((\mathrm{Sym}^m\rho_{E_1,\ell})_{|L'^{H_{i}}}\otimes(\mathrm{Sym}^n\rho_{E_2,\ell})_{|L'^{H_{i}}}\otimes\psi_i)
\end{eqnarray*}
with $\psi_i:=\chi_i\otimes\psi_{|L'^{H_i}}$. By invariance of $L$-series under induction (see for instance \cite[Claim 1.1, pp.13]{Hida}), we get
\begin{multline*}
L\left(s+\frac{m+n}{2},\mathrm{Sym}^m\rho_{E_1,\ell}\otimes\mathrm{Sym}^n\rho_{E_2,\ell}\otimes\chi\otimes\psi\right)=\\ \prod_{i} L\left(s+\frac{m+n}{2},(\mathrm{Sym}^m\rho_{E_1,\ell})_{|L'^{H_i}}\otimes(\mathrm{Sym}^n\rho_{E_2,\ell})_{|L'^{H_i}}\otimes\psi_i\right)^{a_i}
\end{multline*}
Since the subgroups $H_i$ are nilpotent, we infer that the extensions $L'/L'^{H_i}$ are solvable, then one can find a chain of extensions 
$$
L'^{H_i}=L_{0}\subseteq L_{1} \subseteq\cdots\subseteq L_{k}=L'
$$
so that $L_{j+1}/L_{j}$ is cyclic for $0\le j\le k-1$. Then, upon using Arthur--Clozel's base change theorem for automorphic induction successively, in stages, to each of those cyclic extensions, we deduce that $(\mathrm{Sym}^m\rho_{E_1,\ell})_{|L'^{H_i}}$ and $(\mathrm{Sym}^n\rho_{E_2,\ell})_{|L'^{H_i}}$ are cuspidal automorphic over $L'^{H_i}$ . \\
We claim that $\psi_i=\chi_i\otimes\psi_{|L'^{H_i}}$ is a Hecke character of $L'^{H_i}$. In fact, on one hand our hypothesis that $G_1$ is abelian implies that $\psi$  is linear. So, we may regard it via Artin reciprocity as a Hecke character of $F$. Hence, by the base change of $\mathrm{GL}(1)$, $\psi_{|L'^{H_i}}$ is a Hecke character of $L'^{H_i}$. On the other hand, again via Artin reciprocity each of the linear characters $\chi_i$ can be thought of as a Hecke character of $L'^{H_i}$, and so also is $\psi_i$. Therefore, it follows from the functoriality of $\mathrm{GL}(n)\times\mathrm{GL}(1)$ that $(\mathrm{Sym}^n\rho_{E_2,\ell})_{|L'^{H_i}}\otimes\psi_i$ is cuspidal automorphic over $L'^{H_i}$. By the Rankin-Selberg theory ( see for instance \cite[Theorem 9.2, pp. 69]{CKM}) and Shahidi's non-vanishing result \cite[Theorem 5.2]{Shahidi} we deduce that the $L$-function
$$
 L\left(s+\frac{m+n}{2},(\mathrm{Sym}^m\rho_{E_1,\ell})_{|L'^{H_i}}\otimes(\mathrm{Sym}^n\rho_{E_2,\ell})_{|L'^{H_i}}\otimes\psi_i\right)
$$
is invertible if and only if
\begin{equation}
(\mathrm{Sym}^n\rho_{E_2,\ell})_{|L'^{H_i}}\otimes\psi_i\ncong(\mathrm{Sym}^m\rho_{E_1,\ell}^{\vee})_{|L'^{H_i}},\label{eq1}
\end{equation} 
which always holds when $m\neq n$ for dimensional reasons. However, when $m=n$, since $\rho_{E_j,\ell}^{\vee}\cong \rho_{E_j,\ell}\otimes\epsilon_\ell^{-1}$ for $j=1,2$ where $\epsilon_\ell$ stands for the $\ell$-adic cyclotomic character, \eqref{eq1} becomes
$$
(\mathrm{Sym}^n\rho_{E_2,\ell})_{|L'^{H_i}}\otimes\psi_i\ncong(\mathrm{Sym}^n\rho_{E_1,\ell})_{|L'^{H_i}}.
$$
For the sake of contradiction, we assume the contrary of the above displayed formula. By class field theory one may see $\psi_i$ as a Galois character $\mathrm{Gal}(\overline{F}/L'^{H_i})\to\mathbb{C}^{\times}$. Accordingly, there exists a finite extension $N_i/L'^{H_i}$ such that $\psi_{i}|_{\mathrm{Gal}(\overline{F}/N_i)}$ is trivial. Hence 
$$
(\mathrm{Sym}^n\rho_{E_2,\ell})_{|N_i}\cong(\mathrm{Sym}^n\rho_{E_1,\ell})_{|N_i}.
$$
It then follows that the corresponding projective representation of $\rho_{E_2,\ell|N_i}$ would be isomorphic to the projectivizations of $\rho_{E_2,\ell|N_i}$. Hence, there exist a character $\eta : G_{N_i}\longrightarrow\overline{\mathbb{Q}}_\ell^{\times}$ such that $\rho_{E_1,\ell|N_i}\cong\rho_{E_2,\ell|N_i}\otimes\eta$. As $\rho_{E_1,\ell|N_i}$ and $\rho_{E_2,\ell|N_i}$ are de Rham at all places dividing $\ell$ and have the same Hodge-Tate weights (namely $0$ and $1$), we infer that $\eta$ must be of Hodge-Tate weights $0$, thereby, it has finite order. Thus, we get a finite extension $N'_i$ of $N_i$ over which $\rho_{E_1,\ell}$ and $\rho_{E_2,\ell}$  become isomorphic. Consequently, by Flatings theorem we see that $E_1$ and $E_2$ must be isogenous over $N'_i$, which is a contradiction.\\
\textbf{Sub-case 2. $P=\mathrm{End}^0_{\overline{\mathbb{Q}}}(A)=\mathrm{End}^0_{F}(A)$ real quadratic field}\\ 
In this case $A$ is simple. By means of the two embeddings $\lambda_1$ and $\lambda_2$ of  $P\hookrightarrow\overline{\mathbb{Q}_\ell}$ we get the following decomposition $\rho_{A,\ell}=\rho_{A,\lambda_1}\oplus \rho_{A,\lambda_2}$, where the two $\lambda$-adic representations $\rho_{A,\lambda_1}$ and $\rho_{A,\lambda_2}$ of $G_F$ do not become isomorphic over any further extension of $F$. We thus get the rank $2$ weakly compatible system $(\rho_{A,\lambda})_{\lambda}$ of $\lambda$-adic representations of $G_F$ over $P$ which is strongly irreducible (because the image of $\rho_{A,\lambda}$ is open in its Zariski closure $G^{{\rm Zar}}_{\lambda}={\rm GL}_{2}(\overline{\mathbb{Q}}_{\ell})$), totally odd, essentially self-dual and with Hodge-Tate multiset $\{0,1\}$. For every $n,m\ge 0$, one can then apply \thref{Potaut} to get a totally real Galois extension $L'/F$ containing $K^{G_1}$ over which the two weakly compatible systems $(\mathrm{Sym}^n\rho_{A,\lambda})_{\lambda}$ and $(\mathrm{Sym}^m\rho_{A,\lambda})_{\lambda}$ become cuspidal automorphic. So, the invertibility of the $L$-functions
\begin{eqnarray*}
L(s,\rho_{m,n}\otimes\xi)
&=& L\left(s+\frac{m+n}{2},\tilde{\rho}_{m,n}\circ\rho_{A,\ell}\otimes\xi\right)\\
&=& L\left(s+\frac{m+n}{2},\mathrm{Sym}^m\rho_{A,\lambda_2}\otimes\mathrm{Sym}^n\rho_{A,\lambda_1}\otimes\xi\right)
\end{eqnarray*}
for every $\xi\in\mathrm{Irr}(G)$, can be proved along with the arguments of the previous sub-case (see the proof of \cite[Proposition 22]{Johansson}). \\
\textbf{Case 2. $A$ of Galois type $\textbf{B}[C_2]$.}\\
Notice that, for this Galois type we have $\mathrm{End}^0_{F}(A)=\mathbb{Q}$, and there are two sub-cases corresponding to whether $A_{L}$ is simple and $P=\mathrm{End}^0_{\overline{\mathbb{Q}}}(A)=\mathrm{End}^0_{L}(A)$ is a real quadratic field or  $A_{L}$ is isogenous to a direct sum of nonisogenous elliptic curves, each without CM, and $P=\mathrm{End}^0_{\overline{\mathbb{Q}}}(A)=\mathrm{End}^0_{L}(A)=\mathbb{Q}\times\mathbb{Q}$. In what follow we tackle only the former sub-case, as the second one follows the same pattern of the proof of sub-case 1 of the preceding case.\\
Now, we list all the irreducible representations of the Chebotarev-Sato-Tate group $ST$ in this case. Here the identity connected component $ST_A^0$ of the Sato-Tate group is $\mathrm{SU}(2)\times\mathrm{SU}(2)$ which is an index $2$ subgroup of the Sato-Tate group $ST_A=N(\mathrm{SU}(2)\times\mathrm{SU}(2))=\left<\mathrm{SU}(2)\times\mathrm{SU}(2),J\right>$. We shall build up irreducible representations of $ST_A$ from those of $ST_A^0$ which are of the form $\rho_{n,m}:=\mathrm{Sym}^n(\mathrm{St})\otimes \mathrm{Sym}^m(\mathrm{St})$. Since $J_0=\begin{pmatrix}
0 & 1 \\
-1 & 0
\end{pmatrix}\in\mathrm{SU}(2)$, and for every $(A,B)\in ST^0_A$, we have $J(A,B)J^{-1}=(J_0BJ_0^{-1},J_0AJ_0^{-1})$, where $J=\begin{pmatrix}
0 & J_0 \\
-J_0 & 0
\end{pmatrix}$, it follows that $\rho_{m,n}^{J}\cong\rho_{n,m}$. On account of \cite[Lemma 23]{Johansson} we see that $\rho_{n,m}$ extends to two representations $\rho^1_n$, $\rho^2_n$ of $ST_A$ when $n=m$, and for $n\neq m$ the irreducible representations of $ST_A$ should be of the form $\mathrm{Ind}^{ST_A}_{ST_A^0}\rho_{m,n}$. Conventionally we put $\rho_0^1$ as the trivial representation. It then follows that the irreducible nontrivial representations of the Chebotarev-Sato-Tate group $ST$ are $$\mathrm{Ind}^{ST_A}_{ST_A^0}\rho_{m,n}\otimes\xi\;\text{for}\; n>m\ge0,\quad \rho^1_n\otimes\xi \; \text{and}\; \rho^2_n\otimes\xi \;\text{for}\; n\ge 0$$
where $\xi\in\mathrm{Irr}(G)$. As alluded to in \S 2.2, we need to check that the $L$-functions associated to nontrivial irreducible representations of $ST$ are inversible. 

Let us first look at the representations $\mathrm{Ind}^{ST_A}_{ST_A^0}\rho_{m,n}\otimes\xi$. In order to prove anything about their $L$-functions we need to rewrite them  into a convenient form. Let $g'_{\mathfrak{P}}$ denote the normalized image of Frobenius for a prime $\mathfrak{P}$ in $G_L$, then by the interaction of $L$-functions with induction and linear disjointness of $L$ and $K$ (which allows to make the following identification $G=\mathrm{Gal}(KL/L)$) we have 
\begin{eqnarray}\label{eq2}
\nonumber L\left(s,\mathrm{Ind}^{ST_A}_{ST_A^0}\rho_{m,n}\otimes\xi\right)
&=&\prod_{\mathfrak{p}} \det\left(1-\mathrm{Ind}^{ST_A}_{ST_A^0}\rho_{m,n}(g_{\mathfrak{p}})\otimes\xi(\sigma_\mathfrak{p})q^{-s}_{\mathfrak{p}}\right)^{-1},\\ \nonumber
 &=&\prod_{\mathfrak{P}} \det\left(1-\rho_{m,n}(g'_{\mathfrak{P}})\otimes\xi(\sigma'_\mathfrak{P})q^{-s}_{\mathfrak{P}}\right)^{-1},\\ 
&=&L(s,\rho_{m,n}\otimes\xi),
\end{eqnarray}
where $\sigma'_\mathfrak{P}:= \left(\frac{KL/L}{\mathfrak{P}}\right)$ the Artin symbol relative to the extension $KL/L$. Now, we extend $\rho_{m,n}$ to a representation $\tilde{\rho}_{m,n}$ of $G^{\mathrm{Zar}}_{\ell}(A_L)$ the algebraic group coming from $G_L$. Since $G^{\mathrm{Zar}}_{\ell}(A_L)\subset\mathrm{GL}(2)\times \mathrm{GL}(2)$ we make it explicitly by restricting $\mathrm{Sym}^m(\mathrm{St})\otimes\mathrm{Sym}^n(\mathrm{St})$ from $\mathrm{GL}(2)\times \mathrm{GL}(2)$ to $G^{\mathrm{Zar}}_{\ell}(A_L)$. Such as, we get a Galois representation of $G_L$, namely $\tilde{\rho}_{m,n}\circ\rho_{A_L,\ell}$, and \eqref{eq2} becomes
$$L(s,\rho_{m,n}\otimes\xi)=L\left(s+\frac{n+m}{2},\tilde{\rho}_{m,n}\circ\rho_{A_L,\ell}\otimes\xi\right).$$
On the other hand, $\rho_{A_L,\ell}$ split up as sum of two two-dimensional Galois representations $\rho_{A_L,\lambda_1}$ and $\rho_{A_L,\lambda_2}$, that do not become isomorphic over any further extension of $F$, where $\lambda_1$ and $\lambda_2$ denote the two embeddings of $P=\mathrm{End}^0_L(A)$. Hence, we have the following decomposition of the $L$-function
$$L(s,\rho_{m,n}\otimes\xi)=L\left(s+\frac{n+m}{2},\mathrm{Sym}^m\rho_{A_L,\lambda_1}\otimes\mathrm{Sym}^n\rho_{A_L,\lambda_2}\otimes\xi\right).$$
Notice that $L$ can be either a totally real or a CM field. Thus we get the weakly compatible system $(\rho_{A_L,\lambda})_{\lambda}$ of $\lambda$-adic representations of $G_L$ over $P$, which is strongly irreducible (as the image of $\rho_{A_L,\lambda}$ is open in its Zariski closure $G^{{\rm Zar}}_\lambda={\rm GL}_2(\overline{\mathbb{Q}}_\ell)$), totally odd, essentially self-dual and with Hodge-Tate multiset $\{0,1\}$. We infer from \thref{Potaut} that there exists a finite Galois extension $L'/L$ which contains $(KL)^{G_1}$ (which is either a totally real or a CM field according to $L$) so that the two weakly compatible systems $(\mathrm{Sym}^m\rho_{A_L,\lambda})_{\lambda}$ and $(\mathrm{Sym}^n\rho_{A_L,\lambda})_{\lambda}$ become automorphic and cuspidal over $L'$. Set $\tilde{G}:=\mathrm{Gal}(L'/L)$ (temporarily), and write $\xi=\chi\otimes\psi$ with $\chi\in\mathrm{Irr}(G_2)$ and $\psi\in\mathrm{Irr}(G_1)$. As in the previous case, since $L'$ contains $(KL)^{G_1}$, $\chi$ can be seen as a character of $\tilde{G}$ and by Brauer's induction theorem one may write $\chi=\Bigoplus_{i} a_i\mathrm{Ind}_{H_i}^{\tilde{G}}\chi_i$, for some integers $a_i$ and some one-dimensional characters $\chi_i$ of nilpotent subgroups $H_i$ of $\tilde{G}$. Following the steps of the previous case {\it mutatis mutandi} we see that it turns out to prove that $$
(\mathrm{Sym}^n\rho_{A_{L},\lambda_2})_{|L'^{H_i}}\otimes\psi_i\ncong(\mathrm{Sym}^m\rho_{A_{L},\lambda_1})_{|L'^{H_i}}.
$$
Since $m\neq n$ (from irreducibility) then a dimensional argument shows that the above displayed formula holds.

Next, we deal with the representations of the form $\rho^{k}_{n} \otimes  \xi$ with $n\ge 1$ and $k\in\{1,2\}$. Extending $\rho_{n}^k$ to an algebraic representation $\tilde{\rho}^{k}_{n}$ of $G_{\ell}^{\text{Zar}}\subset \left<\mathrm{GL}(2)\times \mathrm{GL}(2),J\right>$ by restricting $\mathrm{Sym}^n(\mathrm{St})\otimes\mathrm{Sym}^n(\mathrm{St})$ and leaving the image of $J$ alone. Then we have
\begin{equation}\label{eq3}
L\left(s,\rho^{k}_{n}\otimes\xi\right)=L\left(s+n,\tilde{\rho}^{k}_{n}\circ\rho_{A,\ell}\otimes\xi\right).
\end{equation}
As before we write $\xi=\psi\otimes\chi$ with $\chi\in\mathrm{Irr}(G_2)$ and $\psi\in\mathrm{Irr}(G_1)$. Let $L'$ be the field from the previous paragraph which is allowed to be Galois over $F$ (since it is so over $\mathbb{Q}$). Our choice of $L'$ to be containing $(LK)^{G_1}$ enables us to inflate $\chi$ to a character of $\tilde{G}:=\mathrm{Gal}(L'/F)$. By Brauer's theorem we have $\chi=\Bigoplus_i a_i\mathrm{Ind}_{H_i}^{\tilde{G}}\chi_i$.  Therefore, the invisibility of the $L$-function \eqref{eq3} follows from that of
\begin{equation}\label{eq4}
L\left(s+n,(\tilde{\rho}^{k}_{n}\circ\rho_{A,\ell})_{|L'^{H_i}}\otimes\psi_{i}\right).
\end{equation}
So two situations may occur, either $L\subseteq L'^{H_i}$ or $L\nsubseteq L'^{H_i}$. In the former case \eqref{eq4} equals 
$$L\left(s+n,(\mathrm{Sym}^n\rho_{A_{L},\lambda_1})_{|L'^{H_i}}\otimes(\mathrm{Sym}^n\rho_{A_{L},\lambda_2})_{|L'^{H_i}}\otimes\psi_i\right)$$ 
to which the same types of arguments as those in the paragraph after \eqref{eq1} apply, yielding its  invertibility.\\
In the case  when $L\nsubseteq L'^{H_i}:=L_{i}$, we consider $L_i'=LL_i$ which is a quadratic CM extension of $L_{i}$. We then have
$$
{\rm Ind}_{G_{L'_{i}}}^{G_{L_{i}}}(\tilde{\rho}^{1}_{n}\circ\rho_{A,\ell})\otimes\psi_i)_{|L'_{i}}\cong  (\tilde{\rho}^{1}_{n}\circ\rho_{A,\ell})_{|L_{i}}\otimes\psi_i\oplus ((\tilde{\rho}^{1}_{n}\circ\rho_{A,\ell})_{|L_{i}}\otimes\psi_i)^{\delta}
$$
where $\delta$ is the non-trivial character of  $\mathrm{Gal}(L/F)$. Taking into account that $(\tilde{\rho}^{1}_{n}\circ\rho_{A,\ell})_{|L_i}\otimes\delta=(\tilde{\rho}^{2}_{n}\circ\rho_{A,\ell})_{|L_i}$, we get the following identity 
$$
L\left(s+n,(\tilde{\rho}^{1}_{n}\circ\rho_{A,\ell})_{|L_i}\otimes\psi_{i}\right)L\left(s+n,(\tilde{\rho}^{2}_{n}\circ\rho_{A,\ell})_{|L_i}\otimes\psi_{i}\right)=L\left(s+n,(\tilde{\rho}^{k}_{n}\circ\rho_{A,\ell})_{|L'_{i}}\otimes\psi_{i|L'_{i}}\right)
$$
Since $L\subset L'_{i}$ we have $(\tilde{\rho}^{k}_{n}\circ\rho_{A,\ell})_{|L'_{i}}=(\mathrm{Sym}^n\rho_{A_L,\lambda_1})_{|L'_{i}}\otimes(\mathrm{Sym}^n\rho_{A_{L},\lambda_2})_{|L'_{i}}$. On the other hand by Arthur--Clozel's cyclic base change we see that the weakly compatible system $(\mathrm{Sym}^n\rho_{A_L,\lambda})_{\lambda}$ is cuspidal automorphic over $L_i'$. It thus follows from the same types of arguments as those in the paragraph just after \eqref{eq1}, that the $L$-function in the right-hand side of the above factorization, is invertible. Accordingly, it suffices to show the non-vanishing of the $L$-functions \eqref{eq4} on $\re(s)\ge 1$ to deduce their invertibility. Interpreting them as the Asai/twisted Asai $L$-function relative to the quadratic extension $L'_i/L_i$ associated to the automorphic representation, coming from $(\mathrm{Sym}^n\rho_{A_L,\lambda}\otimes\psi_i)_{\lambda}$, via the Langlands-Shahidi method. It may be concluded  from \cite[Theorem 4.3]{GS15} that these $L$-functions are non-vanishing on $\re(s)\ge 1$, as desired.

Finally, we investigate the inversibility of the $L$-functions associated to $\rho_{0}^{k}\otimes\xi$ with $k\in\{1,2\}$. Let's start with $L(s,\rho_{0}^{2}\otimes\xi)$. The representation $\rho_{0}^{2}$ is trivial on $ST^{0}_{A}$, so it can be thought of as a representation of $ST_{A}/ST^{0}_{A}\cong \mathrm{Gal}(L/F)$. Therefore, $L(s,\rho_{0}^{2}\otimes\xi)$ is the twisted Hecke $L$-function associated with the nontrivial character $\delta$ of $\mathrm{Gal}(L/F)$. Writing $\xi=\psi\otimes\chi$ with $\chi\in\mathrm{Irr}(G_2)$ and $\psi\in\mathrm{Irr}(G_1)$. Brauer's induction theorem gives the following decomposition $\chi=\Bigoplus_i a_i\mathrm{Ind}_{H_i}^{G}\chi_i$ for some integers $a_i$ and one-dimensional characters $\chi_i$ of nilpotent subgroups $H_i$ of $G$. It follows that
$$
L(s,\rho_{0}^{2}\otimes\xi)=L(s,\delta\otimes\xi)=\prod_{i} L(s,\delta_{|K^{H_i}}\otimes\psi_i)^{a_i}\quad\text{with}\quad \psi_i:=\psi_{|K^{H_i}}\otimes\chi_i.
$$
As we have seen, $\psi_i$ may be regarded as a Hecke character of $K^{H_i}$. Thereby, $L(s,\delta_{|K^{H_i}}\otimes\psi_i)$ is inversible, and also is $L(s,\rho_{0}^{2}\otimes\xi)$. The inversibility of the $L$-function $L(s,\rho_{0}^{1}\otimes\xi)$ with $\xi$ nontrivial, may be proved similarly using Brauer's induction theorem. Which achieve the proof of the proposition.
\end{proof}
\subsection{Abelian surfaces of type $\mathbf{C}$}
In this section we shall investigate the case of abelian surfaces of absolute type $\mathbf{C}$. In fact, we prove the following equidistribution result.
\begin{pro}\thlabel{thm2}
Let $A$ be an abelian surface over a totally real number field $F$ of type $\mathbf{C}$ and let $K/F$ be a Galois extension of finite degree with Galois group $G\cong G_1\times G_2$, such that $G_1$ is abelian and $K^{G_1}/F$ is totally real, moreover, we assume that $L$ and $K$ are linearly disjoint over $F$, then \thref{CST} holds for $A$.
\end{pro}
\begin{proof}
By virtue of \cite{FKRS}[(ii) of Proposition 4.5] $A$ is isogenous over $F$ to a product $E_1\times E_2$ of (necessarily nonisogenous) elliptic curves, one with CM defined over $L$, say $E_1$, and the other without CM, say $E_2$.  The only Galois type which can arises is $\mathbf{C}[C_2]$, because $F$ is totally real (cf.,\cite[\S 4.7, pp 1419]{FKRS}).\\
The Chebotarev-Sato-Tate group is given by $ST=ST_A\times G$ with $ST_A=N(\mathrm{U}(1))\times\mathrm{SU}(2)$ and $N(\mathrm{U}(1))$ is the normalizer of $\mathrm{U}(1)$ in $\mathrm{SU}(2)$. Let us first classify all irreducible non-trivial representations of $ST$. Recall that $\mathrm{U}(1)$ is a subgroup of $N(\mathrm{U}(1))$ of index two, $N(\mathrm{U}(1))=\mathrm{U}(1)\sqcup J_{0} \mathrm{U}(1)$ where  $J_0=\begin{pmatrix}
0 & 1 \\
-1 & 0
\end{pmatrix}$, and the homomorphisms
$$
\phi_m : \begin{array}[t]{lll}
               \mathrm{U}(1) &\longrightarrow & \mathrm{GL}_1(\mathbb{C}) \\
               \begin{pmatrix}
          u   & 0\\
          0   &  \bar{u}  
\end{pmatrix} &\longmapsto &  u^m \\
               \end{array}
$$
with $m\in \mathbb{Z}$ form the complete list of irreducible representations of $\mathrm{U}(1)$. Note, indeed, that for $D\in \mathrm{U}(1)$
$$
\phi_m^{J_0}(D)=\phi_m(J_0 D J_0^{-1})=\phi_m(D^{-1})=\phi_{-m}(D).
$$
It follows that $\phi_m^{J_0}\cong \phi_{-m}$. Accordingly, by \cite[Lemma 23]{Johansson}, for $m\in \mathbb{Z}\setminus\{0\}$ the irreducible representations of $N(\mathrm{U}(1))$ are of the form $\mathrm{Ind}^{N(\mathrm{U}(1))}_{\mathrm{U}(1)}\phi_m$ and $\phi_0$ extends to two representations $\rho_0^0$ and $\rho_0^1$, where we take $\rho_0^0$ to be the trivial representation. Hence, the complete list of irreducible representations of $ST$ is
$$
\rho_{m,n}\otimes\xi \quad\text{and}\quad\rho_0^\varepsilon\otimes \mathrm{Sym}^n(\mathrm{St})\otimes\xi, \quad\text{for}\quad \varepsilon\in\{0,1\}, \quad \text{ $m\ge 1$,  $n\ge 0$ and $\xi\in \mathrm{Irr}(G)$},
$$
with $\rho_{m,n}:=\mathrm{Ind}^{N(\mathrm{U}(1))}_{\mathrm{U}(1)}\phi_m\otimes \mathrm{Sym}^n(\mathrm{St})$, where ${\rm Sym}^n({\rm St})$ is the $n$-th symmetric power of the  standard 2-dimensional representation of ${\rm GL}(2)$. \\ 
We first prove the invertibility of the $L$-functions associated to the representations $\rho_{m,n}\otimes\xi$ for $m\ge 1$ and $n\ge 1$. Note that, the $\ell$-adic representation $\rho_{A,\ell}$ attached to $A$ decomposes into a direct sum $\rho_{A,\ell}=\rho_{E_1,\ell}\oplus\rho_{E_2,\ell}$ of the two $\ell$-adic representations $\rho_{E_1,\ell}$ and $\rho_{E_2,\ell}$ of $G_F$ attached to the two elliptic curves $E_1$ and $E_2$ respectively. Since $E_1$ has CM defined over $L$, its Galois representation is induced from some family of algebraic Hecke characters $(\varphi_\ell)_\ell$, that is, $\rho_{E_1,\ell}=\mathrm{Ind}_{G_L}^{G_F}\varphi_\ell$. Next, we extend $\rho_{n,m}$ to a representation $\tilde{\rho}_{n,m}$ of $G_\ell^{\text{Zar}}$. To make it rigorous and not just formal, since $G_\ell^{\mathrm{Zar}}\subset N(\mathrm{U}(1))\times \mathrm{GL}(2)$ we do so by restricting $\mathrm{Ind}^{N(\mathrm{U}(1))}_{\mathrm{U}(1)}\phi_m\otimes \mathrm{Sym}^n(\mathrm{St})$ to $G_\ell^{\text{Zar}}$. It then follows 
\begin{eqnarray}
\nonumber L\left(s,\rho_{m,n}\otimes\xi\right)
&=&\prod_{\mathfrak{p}} \det\left(1-\rho_{m,n}(g_\mathfrak{p})\otimes\xi(\sigma_\mathfrak{p})q^{-s}_{\mathfrak{p}}\right)^{-1},\\ \nonumber
 &=&\prod_{\mathfrak{p}} \det\left(1-\tilde{\rho}_{m,n}\circ\rho_{A,\ell}(\mathrm{Frob}_\mathfrak{p}))\otimes\xi(\sigma_\mathfrak{p})q^{-s-(m+n)/2}_{\mathfrak{p}}\right)^{-1},\\ \nonumber
&=&L\left(s+(m+n)/2,\mathrm{Ind}_{G_L}^{G_F}\varphi_\ell^m\otimes\mathrm{Sym}^n\rho_{E_2,\ell}\otimes\xi\right). 
\end{eqnarray}
Let ${\rm Frob}_{\mathfrak{P}}$ be the Frobenius element at a prime $\mathfrak{P}|\p$ in $G_L$. Taking into account that $\mathrm{Ind}_{G_L}^{G_F}\varphi_\ell^m\otimes\mathrm{Sym}^n\rho_{E_2,\ell}\cong \mathrm{Ind}_{G_L}^{G_F}(\varphi_\ell^m\otimes\mathrm{Sym}^n\rho_{E_2,\ell|L})$, and the identification $G=\mathrm{Gal}(KL/L)$ we have
\begin{eqnarray}
\nonumber L\left(s,\rho_{m,n}\otimes\xi\right)
&=& \prod_{\mathfrak{p}} \det\left(1-\mathrm{Ind}_{G_L}^{G_F}(\varphi_\ell^m\otimes\mathrm{Sym}^n\rho_{E_2,\ell|L})(\mathrm{Frob}_\mathfrak{p}))\otimes\xi(\sigma_\mathfrak{p})q^{-s-(m+n)/2}_{\mathfrak{p}}\right)^{-1},\\ \nonumber
 &=&\prod_{\mathfrak{P}|\mathfrak{p}} \det\left(1-(\varphi_\ell^m\otimes\mathrm{Sym}^n\rho_{E_2,\ell|L})(\mathrm{Frob}_\mathfrak{P})\otimes\xi(\sigma'_\mathfrak{P})q^{-s-(m+n)/2}_{\mathfrak{P}}\right)^{-1},\\ \nonumber
&=& L\left(s+(m+n)/2,\varphi_\ell^m\otimes\mathrm{Sym}^n\rho_{E_2,\ell|L}\otimes\xi\right), \nonumber
\end{eqnarray}
where $\sigma'_\mathfrak{P}:= \left(\frac{KL/L}{\mathfrak{P}}\right)$ the Artin symbol relative to the extension $KL/L$. Since $(\rho_{E_2,\ell|L})_{\ell}$ is a weakly compatible system of $\ell$-adic representations of $G_L$ which is strongly irreducible, totally odd, essentially self-dual and with Hodge-Tate multiset $\lbrace 0,1\rbrace$ by virtue of \thref{Potaut} one can find a finite Galois extension $L'/L$ containing $(LK)^{G_1}$ over which the system $(\varphi_\ell^m\otimes\mathrm{Sym}^n\rho_{E_2,\ell|L})_{\ell}$ is cuspidal automorphic. Now, we set $\tilde{G}=\mathrm{Gal}(L'/L)$  and write $\xi=\psi\otimes\chi$ where $\psi\in\mathrm{Irr}(G_1)$ and $\chi\in\mathrm{Irr}(G_2)$, and the invertibility of the $L$-function follows as before by making use of Brauer's theorem, Arthur--Clozel's solvable base change, and functoriality of $\mathrm{GL}(n+1)\times\mathrm{GL}(1)$. 

For $m\ge 1$, $n=0$, we have
$$
L(s,\rho_{m,0}\otimes\xi)=L(s+m/2,\varphi^{m}_{\ell}\otimes\xi).
$$
So, the invertibility follows in much the same way as in the last paragraph of the proof of \thref{thm1}. 

We are left with the representations $\rho_0^\varepsilon\otimes \mathrm{Sym}^n(\mathrm{St})\otimes\xi$ for $n\ge 0$ and $\varepsilon\in\{0,1\}$. First, for $n=0$, the invertibility of the $L$-function $L(s,\rho_0^\varepsilon\otimes\xi)$ is similar in spirit to the last paragraph of the proof of \thref{thm1}. For $n\ge 1$, we have
$$L(s,\rho_0^\varepsilon\otimes \mathrm{Sym}^n(\mathrm{St})\otimes\xi)=L(s,\delta\otimes \mathrm{Sym}^n\rho_{E_2,\ell}\otimes\xi),$$
where $\delta$ is a character of ${\rm Gal}(L/F)$. By \thref{Potaut} we get a Galois, totally real extension $L'$ containing $K^{G_1}$  over which the system $(\delta\otimes\mathrm{Sym}^n\rho_{E_2,\ell|L})_{\ell}$ is cuspidal automorphic. Now, the invertibility follows as usual from Brauer's theorem and functoriality of $\mathrm{GL}(n+1)\times\mathrm{GL}(1)$.
\end{proof}
\subsection{Abelian surfaces of type $\mathbf{E}$}
In this section we shall be concerned with the Chebotarev-Sato-Tate distribution on abelian surfaces with quaternionic multiplication. 
\begin{pro}\thlabel{thm3}
Let $A$ be an abelian surface over a totally real number field $F$ of Galois type $\mathbf{E}[C_n]$ for $n=1,3,4,6$ and $\mathbf{E}[C_2,\mathbb C]$ or $\mathbf{E}[C_2,\mathbb R\times \mathbb R]$ when $n=2$ and let $K/F$ be a Galois extension of finite degree with Galois group $G\cong G_1\times G_2$, such that $G_1$ is abelian and $K^{G_1}/F$ is totally real, then \thref{CST} holds for $A$.
\end{pro}
\begin{proof}
Firstly, notice that by \cite[Proposition 4.7]{FKRS}, $P={\rm End}^{0}_{F}(A)$ is an imaginary quadratic field. Accordingly, the $\ell$-adic Galois representation $\rho_{A,\ell}$ of $A$ splits up as direct sum $\rho_{A,\ell}=\rho_{A,\lambda}\oplus\rho_{A,\bar{\lambda}}$, of two 2-dimensional $\lambda$-adic representations of $G_F$, by means of the two embeddings $\bar{\lambda},\,\lambda: P\hookrightarrow\overline{\mathbb{Q}_\ell}$. At this point, one needs to break the proof into two cases.\\
\textbf{Case 1. $A$ of Galois type $\mathbf{E}[C_n]$ with $n=1,3,4,6$ or $\mathbf{E}[C_2,\mathbb C]$.}\\ The corresponding Chebotarev-Sato-Tate groups of these Galois types are $ST=ST_A\times G$ with
$$
ST_A=E_n:=\langle {\rm SU}(2),\Delta_n\rangle, \quad\text{for}\; n=1,2,3,4,6
$$
where $\Delta_n:=\begin{pmatrix}
          e^{i\pi/n}\cdot I   & 0\\
          0   &  e^{-i\pi/n}\cdot I 
\end{pmatrix}\in {\rm SUp}(4).$ The irreducible representations of $ST_A\cong {\rm SU}(2)\times \mu_{2n}$ are $\rho_{k}:={\rm Sym}^k({\rm St})\otimes\eta$ where $k\in\mathbb{N}$, ${\rm Sym}^k({\rm St})$ is the $k$-th symmetric power of the standard 2-dimensional representation of ${\rm GL}(2)$ and $\eta$ is an irreducible character of $\mu_{2n}$  such that $\eta(-I)=(-I)^k$. Hence, the irreducible representations of $ST$ are of the form $\rho_{k}\otimes \xi$ with $\xi\in{\rm Irr}(G)$.\\ 
By virtue of \cite[\S 3]{Ribet}, we have $\rho_{A,\bar{\lambda}}\cong\rho_{A,\lambda}\otimes\varepsilon$ for some character $\varepsilon : G_{F}\longrightarrow P^{\times}$. Note that, as ${\rm End}^{0}_{L}(A)$ is a quaternion algebra $\varepsilon$ must be a character of ${\rm Gal}(L/F)$, hence its order is $n$. Therefore, the representation $\rho_{A,\ell}$ decomposes into a tensor product $s_{A,\lambda}\otimes\delta$ of a $2$-dimensional representation $s_{A,\lambda}:=\rho_{A,\lambda}\otimes\varepsilon^{\frac{1}{2}}$ and an Artin representation $\delta:=\varepsilon^{-\frac{1}{2}}\oplus\varepsilon^{\frac{1}{2}}$ where $\varepsilon^{\frac{1}{2}}$ denotes an arbitrary square root of $\varepsilon$. Furthermore, the Zariski-closure of the image of $s_{A,\lambda}$ is ${\rm GL}_2(\overline{\mathbb{Q}}_\ell)$, and the image of $\delta$ is nothing but $\mu_{2n}$.\\
Consequently, $x_\mathfrak{p}$ is the conjugacy class of $g'_{\mathfrak{p}}q_{\mathfrak{p}}^{-1/2}\otimes h_{\mathfrak{p}}\otimes\sigma_{\mathfrak{p}}$ where $g'_{\mathfrak{p}}:=\iota\circ s_{A,\lambda}({\rm Frob}_\mathfrak{p}))$ and $h_{\mathfrak{p}}:=\delta({\rm Frob}_\mathfrak{p})$. Therefore, the $L$-function of $\rho_{k}\otimes\xi$ is
\begin{eqnarray*}
L(s,\rho_{k}\otimes\xi)&=&\prod_{\mathfrak{p}} \det\left(1-\mathrm{Sym}^k(g'_{\mathfrak{p}})\otimes\eta(h_{\mathfrak{p}})\otimes\xi(\sigma_\mathfrak{p})q_{\mathfrak{p}}^{-s-k/2}\right)^{-1}\\
&=& \prod_{\mathfrak{p}} \det\left(1-\mathrm{Sym}^k(\iota( \rho_{A,\lambda}({\rm Frob}_\mathfrak{p}))\otimes\eta_{k}({\rm Frob}_\mathfrak{p})\otimes\xi(\sigma_\mathfrak{p})q_{\mathfrak{p}}^{-s-k/2}\right)^{-1}\\
&=& L(s+k/2, {\rm Sym}^k \rho_{A,\lambda}\otimes\eta_{k}\otimes\xi)
\end{eqnarray*}
where $\eta_{k} : G_{F}\longrightarrow \mathbb{C}^{\times}$, $\eta_{k}(g):=\varepsilon^{\frac{k}{2}}(g)\otimes\eta(\delta(g))$ for every $g\in G_F$.\\
At this point, we need to check that these $L$-functions are invertible. Note that, $(\rho_{A,\lambda})_{\lambda}$ is a rank $2$ weakly compatible system of $\lambda$-adic representations of $G_F$ which is strongly irreducible, essentially self-dual, totally odd and with Hodge-Tate multiset $\lbrace 0,1\rbrace$, so, we invoke \thref{Potaut} to get a finite Galois extension $L'/F$ containing $K^{G_1}$ such that $({\rm Sym}^k \rho_{A,\lambda})_{|L'}$ is cuspidal automorphic on the other hand we may interpret $\eta_k$ as a Hecke character of $F$ via Artin reciprocity, hence, by the functoriality of  $\mathrm{GL}(k+1)\times\mathrm{GL}(1)$ we deduce that $({\rm Sym}^k \rho_{A,\lambda}\otimes\eta_{k})_{|L'}$ is cuspidal automorphic. As before, writing $\xi=\chi\otimes\psi$ for $\chi\in\mathrm{Irr}(G_2)$ and $\psi\in \mathrm{Irr}(G_1)$. Our assumptions, allows to regard $\chi$ as a character of $\tilde{G}:={\rm Gal}(L'/F)$ and the invertibility of the $L$-function follows as before by making use of Brauer's theorem, Arthur--Clozel's solvable base change, and functoriality of $\mathrm{GL}(k+2)\times\mathrm{GL}(1)$. In the case when $k=0$ and $\eta$ non-trivial the above $L$-function is merely the twisted Hecke $L$-function associated with a non-trivial character of ${\rm Gal}(L/F)$ and its inversibility may be proved as in the last paragraph of the proof of \thref{thm1}.\\
\textbf{Case 2. $A$ of Galois type $\mathbf{E}[C_2,\mathbb R\times \mathbb R]$.}\\ 
In this case the corresponding Sato-Tate group is
$$
ST_A=\langle {\rm SU}(2),J\rangle,
$$ 
where  $J=\begin{pmatrix}
0 & J_0 \\
-J_0 & 0
\end{pmatrix}$ and $J_0=\begin{pmatrix}
0 & 1 \\
-1 & 0
\end{pmatrix}$. Since $J$ commute with ${\rm SU}(2)$, it follows that $ST_A\cong {\rm SU}(2)\times C_2$. Hence, the irreducible representations of  $ST$ in this case are of the form ${\rm Sym}^k({\rm St})\otimes\eta\otimes\xi$, where ${\rm Sym}^k({\rm St})$ is the $k$-th symmetric power of the standard representation of ${\rm GL}(2)$, $\eta$ is an irreducible character of $C_{2}$ and $\xi\in {\rm Irr}(G)$. As in the preceding case we decompose the $\ell$-adic Galois representation $\rho_{A,\ell}$ of $A$ into a tensor product $\rho_{A,\lambda}\otimes\delta$ with $\delta:=\textbf{1}\oplus\varepsilon$ with $\varepsilon : G_{F}\longrightarrow P^{\times}$ is the Galois character obtained from Ribet's theorem. The same reasoning of the previous case applies to conclude the proof.
\end{proof}
Next, we deal with abelian surfaces $A$ over a totally real number field $F$ with Galois types $\mathbf{E}[D_{2n}]$, for $n=2,3,4,6$. From \cite[Proposition 4.7]{FKRS} there exist a quadratic extension $E/F$ such that $A_{E}$ has Galois type $\mathbf{E}[C_n]$ for $n=1,3,4,6$ and $\mathbf{E}[C_2,\mathbb C]$ or $\mathbf{E}[C_2,\mathbb R\times \mathbb R]$ when $n=2$.
\begin{pro}\thlabel{thm4}
Let $A$ be an abelian surface over a totally real number field $F$ of Galois type $\mathbf{E}[D_{2n}]$, for $n=2,3,4,6$ and let $K/F$ be a Galois extension of finite degree with Galois group $G\cong G_1\times G_2$, such that $G_1$ is abelian and $K^{G_1}/F$ is totally real, moreover, we assume that $E$ and $K$ are linearly disjoint over $F$, then \thref{CST} holds for $A$.
\end{pro}
\begin{proof}
The corresponding Sato-Tate groups of these Galois types are
$$
ST_A=\langle {\rm SU}(2),\Delta_n,J\rangle,\qquad\text{with}\;\;n=2,3,4,6.
$$
Let's start by describing their irreducible representations. Taking into account that $ST_A\cong {\rm SU}(2)\times D_{2n}/\langle(-I,-I)\rangle$, it follows that the irreducible representations under consideration are of the form $\rho_{k}:={\rm Sym}^k({\rm St})\otimes\eta$ where $k\in\mathbb{N}$,  as before ${\rm Sym}^k({\rm St})$ stands for the $k$-th symmetric power of the standard representation of ${\rm SU}(2)$ and $\eta$ is an irreducible representation of $D_{2n}$ of matching parity, i.e, $\eta(-I)=(-I)^k$. Hence, the irreducible representations of the Chebotarev-Sato-Tate group $ST$ are of the form $\rho_k\otimes\xi$ with $\xi\in{\rm Irr}(G).$\\
From \cite[Proposition 4.7]{FKRS} we know that $P={\rm End}^{0}_{E}(A)$ is an imaginary quadratic field. Hence, the $\ell$-adic representation $\rho_{A_E,\ell}$ decomposes as a direct sum $\rho_{A_{E},\lambda}\oplus\rho_{A_{E},\bar{\lambda}}$ of two 2-dimensional $\lambda$-adic representations of $G_E$ by means of the two embeddings $\bar{\lambda},\lambda: P\hookrightarrow \overline{\mathbb{Q}}_\ell$. On the other hand, we have $\rho_{A,\ell}\cong {\rm Ind}_{G_E}^{G_F}\rho_{A_E,\lambda}$. It then follows from \cite[Lemma 3.1]{Ribet} that there exists a character of finite order $\varepsilon : G_E\longrightarrow \overline{P}^{\times}$ which is trivial on $G_F$ (as ${\rm End}^{0}_{L}(A)$ is a quaternion algebra) such that $\rho_{A_E,\lambda}^{\sigma}\cong\rho_{A_E,\lambda}\otimes\varepsilon$ where $\sigma\in\mathrm{Gal}(E/F)$ the canonical complex conjugation on $E$. Via Artin reciprocity $\varepsilon$ can be interpreted as a Hecke character of $E$ which is trivial on $\mathbb{A}_{F}^{\times}/F^{\times}$. Then, by Hilbert's theorem 90,
$\varepsilon$ can be written as $\varepsilon=\frac{\gamma^{\sigma}}{\gamma}$ for some Hecke character $\gamma$ of $E$. Then one has  $$(\rho_{A,\lambda}\otimes\gamma^{-1})^{\sigma}\cong \rho_{A,\lambda}\otimes\gamma^{-1}.$$
So, by \cite[Lemma 23]{Johansson}  the representation $\rho_{A,\lambda}\otimes\gamma^{-1}$ extends to a representation $s_{A,\lambda}$ of $G_{F}$. Observe that
 $$\rho_{A,\ell}\cong {\rm Ind}_{G_E}^{G_F}\rho_{A_E,\lambda}\cong {\rm Ind}_{G_E}^{G_F}(s_{A,\lambda|E}\otimes \gamma),$$
which provides us a decomposition of $\rho_{A,\ell}$ into a tensor product $s_{A,\lambda}\otimes\delta$ of a $2$-dimensional representation $s_{A,\lambda}$ such that the Zariski-closure of its image is ${\rm GL}_{2}(\mathbb{Q}_{\ell})$, and an Artin representation $\delta:={\rm Ind}_{G_E}^{G_F}\gamma$ such that its image is precisely $D_{2n}$, as proven in \cite[\S 3.3]{Noah}.\\
Therefore, $x_\mathfrak{p}$ is the conjugacy class of $g'_{\mathfrak{p}}q_{\mathfrak{p}}^{-1/2}\otimes h_{\mathfrak{p}}\otimes\sigma_{\mathfrak{p}}$ where $g'_{\mathfrak{p}}:=\iota\circ s_{A,\lambda}({\rm Frob}_\mathfrak{p}))$ and $h_{\mathfrak{p}}:=\delta({\rm Frob}_\mathfrak{p})$. Therefore, the partial $L$-function of $\rho_{n}\otimes\xi$ is
\begin{eqnarray*}
L(s,\rho_{k}\otimes\xi)&=&\prod_{\mathfrak{p}} \det\left(1-\mathrm{Sym}^k(g'_{\mathfrak{p}})\otimes\eta(h_{\mathfrak{p}})\otimes\xi(\sigma_\mathfrak{p})q_{\mathfrak{p}}^{-s}\right)^{-1}\\
&=& \prod_{\mathfrak{p}} \det\left(1-\mathrm{Sym}^k(\iota( s_{A,\lambda}({\rm Frob}_\mathfrak{p}))\otimes\eta\circ\delta({\rm Frob}_\mathfrak{p})\otimes\xi(\sigma_\mathfrak{p})q_{\mathfrak{p}}^{-s-k/2}\right)^{-1}\\
&=& L(s+k/2, {\rm Sym}^k s_{A,\lambda}\otimes\eta\circ\delta\otimes\xi).
\end{eqnarray*}
The task is now to establish the inversibility of the above $L$-function. Notice that $(s_{A,\lambda})_{\lambda}$ is a rank $2$ weakly compatible system of $\lambda$-adic representations of $G_F$ which inherits the properties of $(\rho_{A,\lambda})_{\lambda}$, consequently, invoking Theorem 2.2 to get a number field $L'$ such that ${\rm Sym}^k s_{A,\lambda|L'}$ cuspidal automorphic. On the other hand, $\eta\circ\delta$ is either a Hecke character or a $2$-dimensional dihedral representation. In the former case the proof can be achieved using the functoriality of ${\rm GL}(k)\times {\rm GL}(1)$ combined with Brauer-Taylor reduction as before. In the case when $\eta\circ\delta$ is a $2$-dimensional dihedral representation, there exist a Hecke character $\kappa$ of $E$, such that $\eta\circ\delta={\rm Ind}_{G_E}^{G_F}\kappa$. Since $E$ and $K$ are linearly disjoint over $F$, we may identify $G$ with ${\rm Gal}(KE/E)$. Let ${\rm Frob}_{\mathfrak{P}}$ be the Frobenius element at a prime $\mathfrak{P}|\p$ in $G_E$, then we have
\begin{eqnarray*}
 L(s+k/2, {\rm Sym}^k s_{A,\lambda}\otimes\eta\circ\delta\otimes\xi)&=&L(s+k/2, {\rm Sym}^k s_{A,\lambda}\otimes{\rm Ind}_{G_E}^{G_F}\kappa\otimes\xi)\\
&=& L(s+k/2, {{\rm Ind}_{G_E}^{G_F}(\rm Sym}^k s_{A,\lambda|E}\otimes\kappa)\otimes\xi)\\
&=&\prod_{\mathfrak{p}} \det\left(1-{{\rm Ind}_{G_E}^{G_F}(\rm Sym}^k s_{A,\lambda|E}\otimes\kappa)({\rm Frob}_{\mathfrak{p}})\otimes\xi(\sigma_\mathfrak{p})q_{\mathfrak{p}}^{-s-k/2}\right)^{-1}\\
&=& \prod_{\mathfrak{P}|\p} \det\left(1-({\rm Sym}^k s_{A,\lambda|E}\otimes\kappa)({\rm Frob}_{\mathfrak{P}})\otimes\xi(\sigma'_\mathfrak{P})q_{\mathfrak{P}}^{-s-k/2}\right)^{-1}\\
&=&  L(s+k/2, {\rm Sym}^k s_{A,\lambda}\otimes\kappa\otimes\xi).
\end{eqnarray*}
where $\sigma'_\mathfrak{P}:= \left(\frac{KL/L}{\mathfrak{P}}\right)$ the Artin symbol relative to the extension $KE/E$. Since $K^{G_1}$ is totally real the field $(EK)^{G_1}$ is either CM or totally real according to $E$. Invoking \thref{Potaut}, we get a number field $L'$ containing $(EK)^{G_1}$ such that $({\rm Sym}^k s_{A,\lambda}\otimes\kappa)_{|L'}$ is cuspidal automrphic. We set $\tilde{G}:={\rm Gal}(L'/E)$, now the proof can be achieved as before by using the usual Brauer's theorem argument and the functionality of ${\rm GL}(k+1)\times {\rm GL}(1)$.
\end{proof}
\paragraph*{Acknowledgments}
I wish to thank Gabor Wiese for drawing my attention to this problem, as well as for their invaluable comments on an earlier version.  I would also like to thank Christian Johansson and Noah Taylor for their helpful conversations.

\Addresses
\end{document}